\documentclass[12pt]{article}
\usepackage{amsfonts,amssymb,amsbsy,amsmath,amsthm,enumerate,verbatim}
\usepackage{mathrsfs}

\usepackage[color]{showkeys}
\definecolor{refkey}{gray}{.75}
\definecolor{labelkey}{gray}{.50}

\usepackage[colorlinks=true]{hyperref}
\hypersetup{urlcolor=blue,citecolor=red}

\newtheorem{theo}{Theorem}[section]
\newtheorem{prop}[theo]{Proposition}
\newtheorem{Lem}[theo]{Lemma}

\newtheorem{remark}[theo]{Remark}

\numberwithin{equation}{section}
\def\<{\langle}
\def \>{\rangle}
\def\intg{\int_{\mathbb{R}^{d}}}
\def\E{\textbf{E}}
\def\gd{\nabla}
\def\R{\mathbb{R}}

\let\phi=\varphi
\let\epsilon=\varepsilon

\begin{document}

\title{Convergence to equilibrium for 
Fokker-Planck equation with a general force 
field}
\author{Mamadou Ndao \\
Laboratoire de Math\'ematiques \\ de Versailles,\\
UVSQ, CNRS, \\
universit\'e Paris Saclay,\\
78035 Versailles, France}

\maketitle
\vspace{0.3 cm}
\begin{abstract}
The aim of this paper is to study the 
convergence of the solution of the 
Fokker-Planck
equation to the associated
 stationary state when  time 
 goes to infinity. The force field which we
consider here is of a general structure, that
is it may not derive from a potential. The 
proof is based on an adequate splitting
$L=B+A$ of the Fokker-Planck operator $L$ and 
the use of  Krein-Rutmann theory. 
\end{abstract}


\tableofcontents

\section{Introduction and main results}
In this paper we study the Fokker-Planck 
equation 
\begin{equation}\label{eq:1.1}
 \Big \{ \begin{array}{l}
 \partial_{t}f=\Delta{f}+{\rm div}(\mathbf{E}f)
 \\
 f(0, x)=f_{0}.
 \end{array}  
\end{equation}
We assume that 
$\left(1+|x|^{2}\right)^{k/p}f_{0} \in 
L^{p}(\mathbb{R}^{d})$ for some $k\geq 0$ and 
$\textbf{E}:\mathbb{R}^{d}\mapsto 
\mathbb{R}^{d}$.
Unlike previous works, we consider a 
general vector field $\textbf{E}$, which does 
not derive necessarily
from a potential. 
Under suitable assumptions on the vector field we establish that $f(t)\rightarrow M(f_{0})G$ with exponential rate as time $t$ goes to infinity. 
The function $G$ stands for the positive
stationary solution of \eqref{eq:1.1} with total mass equal to $1$, and  the mass of a function is defined throughout this paper by 
$$M(g):=
\intg g(x)dx, \qquad \mbox{for all integrable function} \quad g$$ 
We thus generalise similar results obtained by M.P. Gualdani, S. Mischler and C.
Mouhot in \cite{GMM}, where the 
force field $\textbf{E}$ is 
assumed to be of the form 
$\textbf{E}=\nabla \varphi + U$, with  
$\varphi:
\mathbb{R}^{d}\mapsto \mathbb{R}$, and  
$U \in L^{1}_{\rm loc}(\R^{d},\R^{d})$ is 
such that, ${\rm div}(Ue^{-\varphi})=0$.
In that case, it is noticed that $\mu:=e^{-\varphi}$ is a 
stationary 
solution, and a 
Poincar\'e inequality 
$$\int_{\mathbb{R}^{d}}|u-M(u)|^{2}e^{-\varphi}dx
\leq C \int_{\mathbb{R}^{d}}|\nabla u|^{2}
e^{-\varphi}dx$$ 
holds. For instance one may choose a function 
$\varphi$ of the form 
$$\varphi(x)=\frac{1}{\gamma}|x|^{\gamma} 
+ c_{0},$$
where $\gamma\geq 1$ and $c_{0}\in \mathbb{R}$
is such that 
$$\int_{\mathbb{R}^{d}}e^{-\varphi(x)}dx=1,$$
that is $\mu=e^{-\varphi}$ is a 
probability mesure.
Thus under some appropriate conditions on the 
function $\varphi$ they prove that the solution 
$f$ of \eqref{eq:1.1} converges to $ 
M(f_{0})\;\mu$, when $t$ goes to 
infinity, where we denote 
\begin{equation}\label{eq:1.2}
M(f_{0}):=\int_{\mathbb{R}^{d}}
f_{0}(x)d(x)=\intg f_{0}(x)dx.
\end{equation}
\subsection{Hypotheses}

Throughout this paper, we assume that the 
following assumption (\textbf{H0}) holds:
\begin{itemize}
\item[(\textbf{H0})] The vector field 
$\textbf{E}
\in W^{1, \infty}_{\rm loc}(\mathbb{R}^{d}, 
\mathbb{R}^{d}).$ 
\end{itemize}
\medskip

\noindent Furthermore in the statement of our results we 
may require some of the following assumptions:
\begin{itemize}

\item[(\textbf{H1})]
For some constants
$\alpha,\alpha_{2} >0 \geq 0$, and $\beta,\beta_{2} \in \mathbb{R}$  and exponents 
$$1<\gamma \leq \gamma_{2} \leq 2$$ and for all 
$x\in \mathbb{R}^{d},$ the following holds
\begin{equation}\label{eq:1.3}
\alpha |x|^{\gamma}-\beta \leq x\cdot \E \leq \alpha_{2} |x|^{\gamma_{2}} + \beta_{2}.
\end{equation}

\item[(\textbf{H2})]
There exists $\beta_{0}\in \R$
such that the following holds
\begin{equation}\label{eq:1.4}
-\frac{p}{p'}\;{\rm div}
(\textbf{E}) + k \frac{x}{1+|x|^{2}}
 \cdot\textbf{E} \geq \beta_{0}, 
\end{equation}
for all $p\in [1,\;\infty)$ and $p'$ the conjugated exponent defined by $\frac{1}{p} + \frac{1}{p'}=1.$
\item[(\textbf{H3})]
There exists $\omega^{\star} > 0$ and $R>0$ 
large enough, such that for all $k>0$, 
$p\in [2, +\infty)$, 
and for all $x\in \mathbb{R}^{d}$ satisfying 
$|x|> R$, we have 
\begin{equation}\label{eq:1.5} 
-\frac{kd + k(k+d-2)|x|^{2}}{\left(1+|x|^{2}
\right)^{2}} - \frac{p}{p'}
{\rm div}(\textbf{E})+k\frac{x}{1+|x|^{2}}\cdot
\textbf{E} \geq \omega^{\star}.  
\end{equation}
\end{itemize}

An example of such vector fields is given by
$$\E:=\E_{0}+\E_{1},$$ where $$\E_{0}:=\frac{\gd\<x\>^{\gamma}}{\gamma}\quad \mbox{and }\quad \E_{1}\in \mathcal{C}^{1}(\R^{d}; \R^{d}), \quad \E_{1}\rightarrow 0 \quad \mbox {when}\quad |x|\rightarrow +\infty.$$

\subsection{Main results}

We shall denote by 
$L^{2}_{k}$ the space 
$$L^2_{k} := L^{2}_{k}(\R^{d}):=\big \{ f\in L^{2}(\R^{d})
\; ;\; \intg |f(x)|^{2}\<x\>^{k}dx < \infty 
\big \}$$ and    
the linear operator $(L, D(L))$ by 
\begin{equation}\label{eq:1.6}
Lu:={\rm div}( \nabla u + \textbf{E}u),\quad\mbox{for}\quad u\in 
D(L):=\big \{u\in L^{2}_{k}\; ;\;
Lu \in L^{2}_{k} \big \}.
\end{equation}
We shall see that the linear operator 
$L$ generates a $C_{0}$-semigroup  
on $L^{2}_{k}$. 
Now we are in a position to 
state our first result regarding the stationnary solution of \eqref{eq:1.1}.
\begin{theo}\label{thm:GS}
Let $(L,D(L))$ be defined by \eqref{eq:1.6}. 
Assume that hypotheses 
(\textbf{H1}) and (\textbf{H2}) 
hold with $p=2$ in (\textbf{H2}). Then  the 
equation \eqref{eq:1.1} has 
a stationary solution. More precisely
there exists a 
unique function $G \in L^{2}_{k}$, 
strictly positive, such that 
\begin{equation}\label{eq:1.8}
LG=0 \quad \mbox{and} \quad 
\int_{\mathbb{R}^{d}} G(x) dx =1.
\end{equation}
\end{theo}

The following is our main result regarding the evolution equation.
\begin{theo}\label{thm:Cv-GS}
Assume that hypotheses
(\textbf{H1}), (\textbf{H2}) and  
(\textbf{H3}) hold and  let 
$f_{0}\in L^{2}_{k}$ the initial datum of 
\eqref{eq:1.1} be given. Then there 
exists a unique function $f\in C\left([0,\infty)\; ;\; L^{2}_{k}\right)$ 
solution to 
\eqref{eq:1.1}. Moreover, there exist a real number
$\omega >0$ and a constant $C>0$ such that  
\begin{equation}\label{eq:1.9}
\forall\, t\geq 0 \qquad
\|f(t)- M(f_{0})G
\|_{L^{2}_{k}}\leq
C\exp \left(-\omega t\right)
\|f_{0}- M(f_{0})G
\|_{L^{2}_{k}}. 
\end{equation}
\end{theo}

\section{Definitions and notations}

We recall in this section some 
definitions and notations. We denote 
$$\langle x\rangle:=(1+|x|^{2})^{1/2}\quad 
\mbox{and}\quad  
\langle x\rangle^{k}:= \left(1+|x|^{2}
\right)^{\frac{k}{2}}.$$
We define the Lebesgue spaces  
$L^{p}_{k}(\mathbb{R}^{d})$ and the Sobolev 
spaces $H^{1}_{k}(\mathbb{R}^{d})$ as 
follows: 
$$L^{p}_{k}
(\mathbb{R}^d):=\Big \{f 
\in L^{p}(\mathbb{R}^d)\; ; \; 
\int_{\mathbb{R}^{d}} |f(x)|^{p}
\langle x \rangle^{k}dx <\infty  \Big\},$$
for all $p,k\in \R$ such that 
$k>0$ and $p\geq 1.$ We endow these spaces
with their natural norms. 
The scalar product of $L^{2}_{k}$ is defined by
$$(f|g)_{k}:=\int_{\mathbb{R}^{d}}
f(x)\overline{g(x)}\langle x\rangle^{k}dx,$$ 
for all $f$ and $g$ $\in L^{2}_{k}$. 
We set 
$$H^{1}_{k}(\mathbb{R}^{d}):=\Big \{f\in 
L^{2}_{k}(\mathbb{R}^{d})\; ; \;\nabla f\in 
(L^{2}_{k}(\mathbb{R}^{d}))^{d}\Big \}.$$
It is clear that the space 
$C^{\infty}_{c}(\R^{d})$ is dense in 
$L^{p}_{k}$ for $p < \infty$, as well as in $H^{1}_{k}.$

We recall that 
$R_{L}(\lambda) :=(\lambda - L)^{-1} $ 
denotes the resolvent operator of $L$
for some given $\lambda$ such that 
the operator $(\lambda - L)$
has a bounded inverse and $S_{L}(t):=\exp(tL)$ denotes 
the semigroup generated by $L$. 

To prove Theorems \ref{thm:GS} and \ref{thm:Cv-GS},   
our approach is based on the decomposition of 
the operator $L$  as follows: 
for an appropriately 
chosen bounded operator $B$, we shall split
$L$ in the form $L= B+A,$ where the operator 
$A$ is so that there is some $\tau_{0}\in \R$
such that for all $\tau >\tau_{0}$, 
the linear operator
$A-\tau I$ is  $m$-dissipative, 
(See J. Scher and S. Mischler \cite{JM}).
With the above assumptions we shall 
describe the 
appropriate splitting for $L$, and thereby for 
$S_{L}(t).$ Indeed
the mild solution of 
$$\frac{df}{dt}-Lf=\frac{df}{dt}-Af - Bf=0,$$
is given by 
$$f(t)=S_{L}(t)f_{0}\; \mbox{and also by }
\; f(t)=S_{A}(t)f_{0} + 
\int_{0}^{t}S_{A}(t-\tau)BS_{L}(\tau)f_{0}
d\tau.$$
Where $S_{A}(t)=\exp(tA).$
This means that we can write the semigroup 
$S_{L}(t)$ split as follows:
\begin{equation}\label{eq:Rel-Conv}
S_{L}(t)= S_{A}(t)+(S_{A}\ast 
BS_{L})(t),
\end{equation}
where the convolution $\ast$ is defined by 
$$(S_{A}\ast (BS_{L}))(t):= \int^{t}_{0}S_{A}
(t-\tau)BS_{L}(\tau)f_{0}d\tau.$$
Analogously we may consider the equation 
$$\frac{dv}{dt}-Av=0\quad 
\mbox{with initial data}\quad 
v(0, x)=v_{0}(x),$$
and since we have $Av:=Lv-Bv,$ we conclude that
$$\frac{dv}{dt}-Lv=-Bv\quad 
\mbox{with}\quad v(0,x)=v_{0}(x).$$
This implies that 
$$S_{A}(t)v_{0}=S_{L}(t)v_{0}-\int^{t}_{0}
S_{L}(t-\tau)BS_{A}(\tau)v_{0}d\tau.$$
and then one sees again that \eqref{eq:Rel-Conv} can also
be written as follows:
\begin{equation}\label{eq:Rel-Conv-bis}
S_{L}(t)=S_{A}(t)+ (S_{L}\ast BS_{A})(t).
\end{equation}
The remainder of this paper is organized as 
follows. In section 3,  we state  
preliminary results
wich are used in  the sequel. In section 
4, existence of a 
the stationary solution will be investigated.
And  in the last section 5,  we explore the 
stability issue for the evolution equation 
\eqref{eq:1.1} and prove our second main result, Theorem \ref{thm:Cv-GS}.

\section{Preliminary results} 

In this section we discuss the existence 
of a solution for the evolution equation 
\eqref{eq:1.1} in some Banach spaces.  To this end 
we study some properties
of the linear operator $(L,D(L)).$
\begin{prop}\label{thm:Exist-SG}
Assume that (\textbf{H2}) holds with $p=2.$
Consider $(L, D(L))$ defined by \eqref{eq:1.6}.  Then there exists $\lambda_{0} := \lambda_{0}(2)\in 
\R$ such that for all $\lambda\geq 
\lambda_{0}$ and for $f\in L^{2}_{k}$ we have 
$$((\lambda-L)f|f)_{k}\geq 
(\lambda- \lambda_{0})\|f\|_{L^{2}_{k}}.$$ 
That is, the linear operator $(L, D(L))$
is the generator of a 
$C_{0}$-semigroup 
of contraction on $L^{2}_{k}.$
Moreover if $f_{0}\in L^{2}_{k}$ 
is the initial data of 
\eqref{eq:1.1}, there exists $f\in C
\left([0, T]; L^{2}_{k}\right)$ solution of 
\eqref{eq:1.1} associated to $f_{0}$ and 
$f(t):=S_{L}(t)f_{0}.$ 
\end{prop} 
\begin{proof}
We aim to show that there exists $\lambda_{0}
\in \R$ such that for any $\lambda\geq 
\lambda_{0},$ the operator
$L-\lambda I$ is  
$m$-dissipative, more precisely
$$((\lambda-L)f|f)_{k}\geq 
(\lambda- \lambda_{0})\|f\|_{L^{2}_{k}}.$$
For $f_{0}\in L^{2}_{k},$
consider the following problem:
find a function $f\in D(L)$ such that 
\begin{equation}\label{eq:2.1}
-Lf+ \lambda f = f_{0}.
\end{equation}
We begin by showing that for an appropriate 
choice of $\lambda,$ the operator 
$-L+\lambda I$ is coercive. 
Indeed for, 
$\varphi \in C^{\infty}_{c},$ 
we  have
$$(-L\varphi |\varphi)_{k}= 
-\int_{\mathbb{R}^{d}}
{\rm div}(\nabla \varphi(x) + 
\textbf{E}(x)\varphi(x))\varphi(x)\,
\langle x\rangle^{k}dx.$$
Thus integrating by parts we botain
$$(-L\varphi|\varphi)_{k}=\int_{\mathbb{R}^{d}}
(\nabla \varphi (x)+ \textbf{E}(x)\varphi(x))
\cdot \nabla
(\varphi(x)\,\langle x\rangle^{k})dx.$$
Now computing $\nabla(\varphi(x)\, \langle x
\rangle^{k}) = \nabla \varphi (x)\, \langle x\rangle ^{k} +
\varphi(x)\,\nabla \langle x\rangle ^{k}$, and
using this expression in the identity above, we get
$$(-L\varphi |\varphi)_{k}= 
I_{1} +I_{2} +I_{3} + I_{4},$$
where  for convenience we denote
$$I_{1}:=\int_{\mathbb{R}^{d}}
|\nabla \varphi(x)|^{2}\langle x\rangle^{k}dx,
\qquad  
I_{2} := \int_{\mathbb{R}^{d}} \varphi (x)
\nabla \varphi(x)\cdot\E(x)\, \langle x\rangle^{k}dx,$$ 
$$I_{3} := \int_{\mathbb{R}^{d}}  \varphi(x) \nabla\varphi(x) \cdot
\nabla\langle x\rangle^{k}dx, \qquad \mbox{and} \qquad
I_{4} := \intg \varphi(x)^{2}\E(x)\cdot \nabla\langle x \rangle^{k}
dx.$$
We rewrite $I_{2}$ as 
$$I_{2}=\frac{1}{2}\intg \gd(\varphi^{2})
\cdot \E
\<x\>^{k}dx,$$
and then we
integrate by parts to obtain
$$I_{2}=-\frac{1}{2}\intg \varphi^{2}{\rm div}
(\E\<x\>^{k})dx.$$
Since ${\rm div}(\E\<x\>^{k})={\rm div }(\E)
\<x\>^{k} + \E \cdot \gd \<x\>^{k},$ then   
$I_{2}$ becomes: 
$$I_{2} = -\frac{1}{2}\intg \varphi^{2} 
{\rm div }(\E)
\<x\>^{k} -\frac{1}{2} \intg \varphi^{2} 
\E \cdot
\gd\<x\>^{k}dx.$$ 
Summing $I_{2}$ and $I_{4}$ we obtain
$$I_{2}+I_{4}=-\frac{1}{2}\intg 
\varphi^{2}{\rm div}(\E)
\<x\>^{k}dx +\frac{1}{2}\intg \varphi^{2}
\E\cdot 
\gd\<x\>^{k}dx.$$
The term 
$I_{3}$ can be also rewritten as 
$$I_{3}=\frac{1}{2} \intg \gd(\varphi^{2})
\cdot \gd\<x\>^{k}dx,$$ 
which, integrated by parts, yields 
$$I_{3}=-\frac{1}{2} \intg \varphi^{2}
\Delta\<x\>^{k}dx.$$
Summing together $I_{1},$
 $I_{2},$ $I_{3}$ and $I_{4}$ we obtain
\begin{equation}\label{eq:2.2}
 (-L\varphi|\varphi)_{k}=
 \intg |\gd \varphi|^{2}\<x\>^{k}dx +
 \frac{1}{2}
 \intg \varphi^{2}\frac{\Psi_{k}(x)}{\<x\>^{k}
 }\<x\>^{k}dx,
 \end{equation} 
 where  we have set
 \begin{equation}\label{eq:2.3}
 \Psi_{k}(x):= -\Delta \<x\>^{k}- 
 {\rm div}(\E)+ \E\cdot \gd \<x\>^{k}.
 \end{equation}
Using the expressions 
\begin{equation}\label{eq:nabla-x}
\nabla \langle x\rangle^{k}:=
k\frac{x}{1+|x|^{2}}\langle x\rangle^{k},
\end{equation}
and 
\begin{equation}\label{eq:laplacian-x}
\Delta \langle x\rangle^{k}:=
(kd+k(k+d-2)|x|^{2})\frac{1}{(1+|x|^{2})^{2}}
\langle x\rangle^{k},
\end{equation}
we obtain 
$$\frac{\Psi_{k}(x)}{\<x\>^{k}}=-
\frac{kd + k(k + d-2)|x|^{2}}{(1+|x|^{2})^{2}}-
{\rm div}(\E) + 
\frac{kx}{1+|x|^{2}}\cdot \E.$$ 
Thus using hypothesis (\textbf{H2})
with $p=2$ we find
$$\frac{\Psi_{k}(x)}{\<x\>^{k}}\geq 
\beta_{0}-
\frac{kd + k(k+d-2)|x|^{2}}{(1+|x|^{2})^{2}}.$$ 
now setting 
\begin{equation}\label{eq:2.4}
\lambda_{0}:=\max_{x\in \R^{d}} 
[\frac{kd}{1+|x|^{2}}+
\frac{k(k-2)|x|^{2}}{(1+|x|^{2})^{2}}
-\beta_{0}].
\end{equation} One 
obtains the following inequality
\begin{equation}\label{eq:2.5}
(-L\varphi|\varphi)_{k}\geq 
\int_{\mathbb{R}^{d}}
|\nabla \varphi|^{2}\langle x\rangle^{k}dx  
-\lambda_{0}\int_{\mathbb{R}^{d}}
\varphi^{2}\langle x\rangle^{k}dx\geq 0.
\end{equation}
This inequality, together with a density 
argument, implies that $\lambda-L$
is coercive for all
$\lambda>\lambda_{0}.$  This means that 
for all $\lambda>\lambda_{0}$ we have
$$((\lambda -L)\varphi|\varphi)\geq 
\intg |\gd \varphi|^{2}\<x\>^{k}dx +
(\lambda-\lambda_{0})\intg \varphi^{2}
\<x\>^{k}dx.$$ Then an elementary application
of Lax-Milgram lemma in the space 
$H^{1}_{k}(\R^{d})$ implies that for any 
$\lambda>\lambda_{0}$ and any $f_{0}\in 
L^{2}_{k},$ equation \eqref{eq:2.1} has a unique 
solution $f\in H^{1}_{k}.$\\
Now, for $\mu >0$ given and 
$f_{0}\in L^{2}_{k}$ the equation 
\begin{equation}\label{eq:2.6}
[I+\mu (\lambda_{0}-L)]f=f_{0}
\qquad f\;\in D(L) 
\end{equation} 
has a unique solution, since it is equivalent 
to $$(\frac{1}{\mu}+\lambda_{0})f - Lf=
\frac{1}{\mu}f_{0}$$ and as we have 
$\lambda:= \lambda_{0}+\frac{1}{\mu}>
\lambda_{0}.$
Multiplying \eqref{eq:2.6} by $f$ yields 
$$\|f\|^{2}_{L^{2}_{k}}\leq (f_{0}|f)_{k},$$
from wich we deduce that the unique solution of \eqref{eq:2.6} satisfies
$$\|f\|^{2}_{L^{2}_{k}}\leq 
\|f_{0}\|^{2}_{L^{2}_{k}},$$ that is the 
operator $(L-\lambda_{0}I)$ is $m$-dissipative.
This means that the operator
$(L-\lambda I)$ 
generates a $C_0$-semigroup 
of contraction on $L^{2}_{k}$: 
$$\|e^{-t(\lambda_{0}-L)}f_{0}\|_{L^{2}_{k}}
\leq \|f_{0}\|_{L^{2}_{k}}.$$
From this it is classical to deduce that the 
solution of \eqref{eq:1.1} is given by 
$$f(t):=e^{tL}f_{0}=S_{L}(t)f_{0}$$ and that 
$f\in C([0,+\infty);L^{2}_{k}).$ 
Moreover we have 
$$\|e^{tL}f_{0}\|_{L^{2}_{k}}\leq 
e^{\lambda_{0} t}\|f_{0}\|_{L^{2}_{k}}.$$
This completes the proof of the Proposition \ref{thm:Exist-SG}.
\end{proof}

It is useful to note that for any $f_{0}\in L^2_{k}$ with $f_{0} \geq 0$, when the hypothesis (H3) is satisfied and $\lambda > - \omega^\star/2$ the solution of
\begin{equation}\label{eq:Weak-Max}
\lambda f - Lf = f_{0},\qquad f \in D(L),
\end{equation}
satisfies $f \geq 0$.

\begin{Lem}{\bf (Weak maximum principle).} Assume that the hypotheses (\textbf{H0}) and (\textbf{H3}) are satisfied. We have 
\begin{equation}\label{eq:Coercive-1}
(Lf|f^-) \geq \int_{{\Bbb R}^d}|\nabla f^-(x)|^2\, \langle x \rangle^k\,dx + {\omega^\star \over 2} \int_{{\Bbb R}^d}|f^-(x)|^2\, \langle x \rangle^k\,dx,
\end{equation}
where $\omega^\star$ is given by \eqref{eq:1.5}. 
If $f_{0}\in L^2_{k}$ and $f_{0} \geq 0$, then for any $\lambda > - \omega^\star/2$ the solution of \eqref{eq:Weak-Max} exists and satisfies $f \geq 0$.
\end{Lem}

\begin{proof}
Indeed, to see this, we note that for $f\in D(L)$ we clearly have $f^- \in H^1_{k}$ and thus
$$(Lf|f^{-})_{k} = -\int_{{\Bbb R}^d}\nabla f(x)\cdot \nabla f^-(x)\,\langle x \rangle^k\,dx - \int_{{\Bbb R}^d}f(x) \E(x)\cdot\nabla(f^-(x)\,\langle x\rangle^k)\,dx.$$
Now, since we have $$f=f^{+}-f^{-}, 
\qquad \gd f=\gd f^{+}
-\gd f^{-}\qquad \mbox{and}\qquad f^{+}
\gd f^{-}=0,$$ 
in a first step this allows us to see that
\begin{align*}
(Lf|f^{-})_{k}& =\intg |\gd f^{-}|^{2}
\<x\>^{k}dx +\intg f^{-}(x) \nabla f^{-}(x)\cdot \gd 
\<x\>^{k}dx \\
&+ \intg f^{-}(x)\, \E\cdot\gd f^{-}(x)\, \<x\>^{k}dx
+\intg |f^{-}(x)|^{2}\, \E\cdot \gd\<x\>^{k}dx.
\end{align*}
Now integrating by parts in the second and the 
third terms of the identity above we obtain
\begin{align*}
(Lf|f^{-})_{k}& = \intg |\gd f^{-}|^{2}
\<x\>^{k}dx -
\frac{1}{2}\intg |f^{-}|^{2}\Delta \<x\>^{k}
dx \\
&-\frac{1}{2}\intg |f^{-}|^{2}{\rm div}(\E)
\<x\>^{k}dx
+ \frac{1}{2}\intg |f^{-}|^{2}
\E \cdot \gd \<x\>^{k}dx.
\end{align*}
Using the expressions of $\nabla \langle x \rangle^k$ and that of $\Delta \langle x \rangle^k$ on the one hand, and the fact that $\omega^\star$ satisfies \eqref{eq:1.5}, we conclude that
$$(Lf|f^-) \geq \int_{{\Bbb R}^d}|\nabla f^-(x)|^2\, \langle x \rangle^k\,dx + {\omega^\star \over 2} \int_{{\Bbb R}^d}|f^-(x)|^2\, \langle x \rangle^k\,dx,$$
which is precisely \eqref{eq:Coercive-1}.

Also, proceeding as above one can see that for any $f \in H^1_{k}$ we have
\begin{equation}\label{eq:Coercive-H3}
(-Lf|f) \geq \int_{{\Bbb R}^d}|\nabla f(x)|^2\, \langle x \rangle^k\,dx + {\omega^\star \over 2} \int_{{\Bbb R}^d}|f(x)|^2\, \langle x \rangle^k\,dx.
\end{equation}
Therefore, thanks to the Lax-Milgram theorem equation \eqref{eq:Weak-Max} has a unique solution when $\lambda > - \omega^\star/2$. Moreover, when $f_{0} \geq 0$, multiplying \eqref{eq:Weak-Max} by $f^-$ and integrating by parts we have
$$0 \leq (f_{0}|f^-) = -\lambda\|f^-\|_{L^2_{k}}^2 - (Lf|f^-)_{L^2_{k}}.$$
Using \eqref{eq:Coercive-1} it follows that
$$ 0 \leq - \int_{{\Bbb R}^d}|\nabla f^-(x)|^2\, \langle x \rangle^k\,dx - \left(\lambda + {\omega^\star \over 2} \right)\int_{{\Bbb R}^d}|f^-(x)|^2\, \langle x \rangle^k\,dx ,$$
which, since $\lambda + (\omega^\star/2) > 0$ implies $f^- \equiv 0$, that is $f \geq 0$.
\end{proof}

Our next useful result is the fact that the semigroup $S_{L}(t)$ is positivity preserving.

\begin{Lem}\label{lem:Positif}
Let $f_{0}\in D(L)$ be nonnegative and assume that the hypotheses (\textbf{H0}) and (\textbf{H3}) hold. 
Then $S_{L}(t)f_{0} \geq 0$, that is the solution
$f$ of \eqref{eq:1.1} associated to the initial data
$f_{0}$ is nonnegative.
\end{Lem}

\begin{proof}
Assume that $f_{0}\in D(L)$ with $f_{0}\geq 0$, and consider the equation
\begin{equation}\label{eq:2.7}
\partial_{t}f = Lf,\qquad\qquad f(0) = f_{0}.
\end{equation} 
We aim to show that $f(t) \geq 0$ for all $t \geq 0$. To this 
end we consider $f^{-}$ the negative 
part of $f.$ It is clear that since 
$f\in  C^{1}\big([0,\infty);L^{2}_{k})$, we have $f \in D(L)$ and
$f^{-}\in H^{1}_{k}$. Therefore we may multiply \eqref{eq:2.7} by $-f^{-}$, to get
$$-\intg \partial_{t}f f^{-}\<x\>^{k}dx = - (Lf|f^{-})_{k}\leq 
-{\omega^\star \over 2} \|f^{-}\|^{2}_{L^{2}_{k}},$$
thanks to \eqref{eq:Coercive-1}. Since $f^-\partial_{t}f = -\partial_{t}((f^-)^2)/2$ we see that
$$\frac{1}{2}\frac{d}{dt}\intg |f^{-}(x)|^{2}\,\<x\>^{k}dx 
\leq -{\omega^\star \over 2} \|f^-\|_{L^2_{k}}^2,$$
and using Gronwall's lemma we conclude that 
$$ \intg |f^{-}(x)|^{2}
\<x\>^{k}dx \leq \exp(-\omega^\star t) \int_{{\Bbb R}^d}|f_{0}^-(x)|^2\, \langle x \rangle^k\,dx,$$
from which, since $f_{0}^- \equiv 0$, we infer that $f^{-}\equiv 0$, that is $f\geq 0$. 
\end{proof}

\begin{Lem}\label{lem:SG-Lp}
Assume that the hypotheses (\textbf{H0})--(\textbf{H3}) hold for some  
$p\in [2,\infty)$. 
Then there exists $\lambda_{0}(p)\in \R$ such 
that for any $\lambda\geq \lambda_{0}(p)$, 
the semigroup $S_{L}(t)$ generated 
by $(L, D(L))$ is also a $C_0$-semigroup on $L^p_{k}$.
\end{Lem}

\begin{proof}
Let $f_{0}\in L^{p}_{k}$ and assume that $f_{0}\geq 0$. 
We aim to show that for 
$\lambda > \lambda_{0}(p)$
the equation \eqref{eq:1.1} has a unique solution 
$\varphi \in L^{p}_{k}$ such that $L\varphi \in 
L^{p}_{k}.$\\
To this end we  consider the following problem:
find $\varphi\in L^{p}_{k}$ such that,
\begin{equation}\label{eq:2.9}
-L\varphi +\lambda \varphi= f_{0}.
\end{equation}
We begin by observing that for $f_{0}\in 
C^{\infty}_{c}$ and $f_{0}\geq 0,$
the above equation has a unique solution
$\varphi \in D(L)$ provided $\lambda >
\lambda_{0}(p),$ where $\lambda_{0}(p)$ will be 
precised later. We are going to show that 
when $\lambda >\lambda_{0}(p),$ where 
$\lambda_{0}(p)$ is large enough, we have 
$$\|\varphi\|_{L^{p}_{k}}\leq C 
\|f_{0}\|_{L^{p}_{k}}$$ for a constant $C$
independant of $f_{0}.$ Then a standard 
density argument shows that 
for any $f_{0}\in L^{p}_{k}$ and 
$f_{0}\geq 0$ equation \eqref{eq:2.9}
has a unique solution $\varphi \in L^{p}_{k}$ 
such that $L\varphi \in L^{p}_{k}$ and $\phi \geq 0$.

Let $\zeta_{0}\in C^{\infty}_{c}
\big([0,\infty)\big)$ such that 
$$\zeta_{0}(s)=
\Big \{ \begin{array}{l} 1 \qquad \mbox{if} 
\quad 0\leq s\leq 1
 \\
 0\qquad \mbox{if}\quad s\geq 2,
 \end{array} $$
where $0\leq \zeta_{0}\leq 1$ and 
$-2\leq \zeta_{0}'(s)\leq 0.$
For any integer $n\geq 1$ we define
\begin{equation}\label{eq:Def-zeta-n}
\zeta_{n}(x):=\zeta_{0}\left(\frac{|x|}{n}\right).
\end{equation}
Since $f_{0}\in L^2_{k}$, we know that for $\lambda > \lambda_{0}(2)$, where $\lambda_{0}(2)$ is given by Proposition \ref{thm:Exist-SG}, there exists a unique solution $\varphi\in D(L)$ of \eqref{eq:2.9}. Thus we may multiply the latter equation by $\phi^{p-1}\zeta_{n}$ and integrate by parts to obtain: 
\begin{eqnarray}
&\displaystyle \lambda\int \phi^p\,\zeta_{n}\,\langle x\rangle^k\,dx 
+ \int \nabla\phi\cdot\nabla(\phi^{p-1}\zeta_{n}\,&\!\!\!\!\langle x\rangle^k)\,dx +
\int \phi\E\cdot\nabla(\phi^{p-1}\zeta_{n}\,\langle x\rangle^k)\,dx  \nonumber\\
&& = \int f_{0}\,\phi^{p-1}\zeta_{n}\,\langle x\rangle^k\,dx\label{eq:Identity-1}
\end{eqnarray}
In order to make the proof more clear, we are going to treat the second and third integrals of the first line of the above equality separately, and show the Lemma in several steps.

\noindent{\bf Step 1.} The second term in the first line of the identity \eqref{eq:Identity-1} can be written as
\begin{equation}\label{eq:Identity-2}
\int \nabla\phi\cdot\nabla(\phi^{p-1}\zeta_{n}\,\langle x\rangle^k)\,dx = A_{1} + A_{2} + A_{3},
\end{equation}
where we have set
\begin{equation}\label{eq:Def-A1A2}
A_{1} := (p-1)\intg |\gd \varphi|^{2}
\varphi^{p-2}\zeta_{n}\, \<x\>^{k}\,dx,\qquad 
A_{2} := \intg \varphi^{p-1}
\gd \varphi \cdot \gd\zeta_{n}\,
\<x\>^{k}\,dx,
\end{equation}
and 
\begin{equation}\label{eq:Def-A3}
A_{3} :=\intg \varphi^{p-1}
\gd \varphi \cdot\zeta_{n}
\gd \<x\>^{k}dx.
\end{equation}
Regarding $A_{2}$, writing $\phi^{p-1}\nabla\phi$ as $\nabla(\phi^p)/p$ and integrating by parts we have
\begin{equation}\label{eq:A2}
A_{2} = -\frac{1}{p}\intg \varphi^{p}\;
(\Delta \zeta_{n})\, \<x\>^{k}dx - 
\frac{1}{p}\intg \varphi^{p}\;
\gd\zeta_{n}\cdot\gd\<x\>^{k}dx.
\end{equation}
Using the expressions 
$$\gd \zeta_{n}(x):=\zeta'_{0}\left(\frac{|x|}{n}\right)\frac{x}{n|x|}\quad\mbox{and}\quad 
\gd\<x\>^{k}=\frac{kx}{1+|x|^{2}}$$
in \eqref{eq:A2} we obtain finally
\begin{equation}\label{eq:A2-bis}
A_{2} = -\frac{1}{p}\intg \varphi^{p}\;
(\Delta \zeta_{n})\,\<x\>^{k}dx - 
\frac{1}{p}\intg \varphi^{p}\;
\zeta'_{0}\left(\frac{|x|}{n}\right)
\frac{k|x|}{n(1+|x|^{2})}\<x\>^{k}dx.
\end{equation}
Analogously the term $A_{3}$ can be also rewritten and one may check that 
\begin{align}
A_{3} & =\frac{1}{p}\intg \gd (\varphi^{p})
\cdot\zeta_{n}\gd \<x\>^{k}dx\nonumber\\
&=-\frac{1}{p}\intg \varphi^{p}
\gd\zeta_{n}\cdot \gd\<x\>^{k}dx
-\frac{1}{p}\intg \varphi^{p}\zeta_{n}
\Delta \<x\>^{k}dx\nonumber\\
&= -\frac{1}{p}\intg \varphi^{p}
\zeta'_{0}\left(\frac{|x|}{n}\right)
\frac{k|x|}{n(1+|x|^{2})}\<x\>^{k}dx\nonumber\\
&\hskip 2cm -\frac{1}{p}\intg \varphi^{p}\zeta_{n}
 \frac{kd +k(k+d-2)|x|^{2}}{(1+|x|^{2})^{2}}
\<x\>^{k}dx.\label{eq:A3}
\end{align}
Summing the equalities \eqref{eq:A2-bis} and \eqref{eq:A3} we obtain
\begin{align}
A_{2}+A_{3} & = -\frac{1}{p}\intg \varphi^{p}\;
(\Delta \zeta_{n})\,\<x\>^{k}\,dx - \frac{2}{p}\intg 
\varphi^{p}\;\zeta'_{0}\left(\frac{|x|}{n}\right)
\frac{k|x|}{n(1+|x|^{2})}\<x\>^{k}dx\nonumber\\
&\hskip 2cm -\frac{1}{p}\intg \varphi^{p}\zeta_{n}\,
\frac{kd +k(k+d-2)|x|^{2}}{(1+|x|^{2})^{2}}
\<x\>^{k}dx. \label{eq:A2A3}
\end{align}
The facts that 
$$0 \leq \zeta_{n}(x)\leq 1, \qquad
-2 \leq \zeta'_{0}(s)\leq 0, \qquad
|\zeta''_{0}(s)| \leq C,$$
and 
$$ -\Delta \zeta_{n} = {-1 \over n^2}\Delta\zeta_{0}\left({|x| \over n}\right)\geq {-C \over n^2}\, 1_{[n\leq |x| \leq 2n]},$$
allow us to conclude first that 
$$-\intg 
\varphi^{p}\;\zeta'_{0}(\frac{|x|}{n})
\frac{k|x|}{n(1+|x|^{2})}\<x\>^{k}dx\geq 0,$$
and then from \eqref{eq:A2A3} we infer that, since there exists a constant $C >0$ such that 
$$\frac{kd +k(k+d-2)|x|^{2}}{1+|x|^{2}}
\leq C,$$ 
we finally have
\begin{equation}\label{eq:Minor-A2A3}
A_{2}+A_{3} \geq -C
\int_{[n\leq |x|\leq 2n]} \varphi^{p}\,
\frac{\<x\>^{k}}{n^{2}}dx
-C\intg \varphi^{p}\,\frac{\<x\>^{k}}{1+|x|^{2}}
dx.
\end{equation}

\noindent{\bf Step 2.} The third term in the first line of the identity \eqref{eq:Identity-1} can be written as
\begin{equation}
\int \phi\E\cdot\nabla(\phi^{p-1}\zeta_{n}\,\langle x\rangle^k)\,dx = A_{4} + A_{5} + A_{6},
\end{equation}
where we have set
\begin{equation}\label{eq:A4A5}
A_{4} := (p-1)\intg \varphi^{p-1}\, \gd\varphi\cdot
\zeta_{n}\,\E(x)\,\<x\>^{k}dx,\qquad
 A_{5}:=\intg \varphi^{p}\,
\gd \zeta_{n}\cdot 
\E(x)\,\<x\>^{k}dx 
\end{equation}
and 
\begin{equation}\label{eq:A6}
A_{6}:=\intg \varphi^{p}\, \E(x)\cdot \zeta_{n}\,
\gd\<x\>^{k}dx.
\end{equation}
Proceeding as above $A_{4}$ can be rewritten as 
\begin{align*}
A_{4} &= \frac{p-1}{p}\intg \gd(\varphi^{p})
\cdot \zeta_{n}(x)\E(x)\, \<x\>^{k}\,dx\\
&= -\frac{p-1}{p}\intg \varphi^{p}
\zeta_{n}\,{\rm div}(\E(x))\,\<x\>^{k}dx
-\frac{(p-1)}{p}\intg \varphi^{p}\,
\gd\zeta_{n}\cdot\E(x)\, \<x\>^{k}dx\\
&\hskip 2cm -\frac{p-1}{p}\intg\varphi^{p}\,\zeta_{n}\,
\E(x)\cdot \gd\<x\>^{k}dx.
\end{align*}
Summing $A_{4},A_{5},A_{6}$ we get
\begin{align}
A_{4}+A_{5}+A_{6} &=
-\frac{1}{p'}\intg \varphi^{p}\,
\zeta_{n}\, {\rm div}(\E(x))\,
\<x\>^{k}dx\nonumber\\
&\qquad  +\frac{1}{p} \intg \varphi^{p}\,\gd\zeta_{n} \,\cdot\E(x)\, \<x\>^{k}dx\nonumber\\
& \qquad + \frac{1}{p}\intg\varphi^{p}\, \zeta_{n}\,\E(x)
\cdot \gd\<x\>^{k}dx.
\end{align}
Since $0\leq -\zeta'_{0} \leq 2$, using the assumption \eqref{eq:1.3} one checks that for $n$ large enough so that $x\cdot \E(x) \geq 0$ for $|x| \geq n$,
\begin{align*}
\nabla\zeta_{n}(x)\cdot \E(x) = \zeta'_{0}\left(\frac{|x|}{n}\right)
\frac{x}{n|x|}\cdot \E &\geq
\zeta'_{0}\left(\frac{|x|}{n}\right)
\frac{\alpha_{2}|x|^{\gamma_{2}} + \beta_{2}}{n|x|}\,
{\bf 1}_{[n\leq |x|\leq 2n]}
\end{align*}
and so for $n$ large enough we have
\begin{equation}\label{eq:E-gradzeta}
\gd \zeta_{n}(x)\cdot \E(x) \geq -2
\frac{\alpha_{2}|x|^{\gamma_{2}}+\beta_{2}}{
|x|^{2}}{\bf 1}_{[n\leq |x|\leq 2n]}.
\end{equation}
Then using hypothesis 
(\textbf{H2}) and the inequality \eqref{eq:E-gradzeta} we obtain 
$$A_{4}+A_{5}+A_{6}\geq \frac{\beta_{0}}{p}
\intg \varphi^{p}\zeta_{n}\<x\>^{k}dx-
\frac{2}{p}\intg \varphi^{p} 
\frac{\alpha_{2}|x|^{\gamma_{2}}+\beta_{2}}{
|x|^{2}}\, \<x\>^{k}dx.$$ 
Thus, setting  
$$\frac{\Psi_{k,p}}{\<x\>^{k}}:=
\frac{1}{p}\Big [
\frac{C}{1+|x|^{2}}+\frac{2k}{1+|x|^{2}}+
\frac{C}{1+|x|^{2}}
+2\frac{\alpha_{2}|x|^{\gamma_{2}}+ \beta_{2}}{|x|^{2}}\Big ],
$$
and using \eqref{eq:Minor-A2A3}, we have  that 
\begin{equation}\label{eq:A2-A6}
A_{2}+A_{3}+A_{4}+A_{5}+A_{6}\geq 
-\frac{1}{p}\intg \varphi^{p} 
\frac{\Psi_{k,p}}{\<x\>^{k}}\<x\>^{k}dx.
\end{equation}

\noindent{\bf Step 3.} Now if we define 
$$\lambda_{0}(p):=\max_{x\in \R^{d}}
\big [\frac{\Psi_{k,\zeta}}{\<x\>^{k}}
-\frac{\beta_{0}}{p}\big ]$$
we obtain, thanks to \eqref{eq:A2-A6}, \eqref{eq:Def-A1A2}, \eqref{eq:Identity-2} and \eqref{eq:Identity-1}, that
\begin{align*}
\intg f_{0}\varphi^{p-1}\zeta_{n}
&\<x\>^{k}dx  =
\intg (\lambda-L)\varphi\;
\varphi^{p-1}\zeta_{n}\<x\>^{k}dx\\
& \geq 
(\lambda-\lambda_{0}(p))\intg 
\varphi^{p}\zeta_{n}\<x\>^{k} dx
+(p-1)\intg |\gd \varphi|^{2}\varphi^{p-2}
\zeta_{n}\<x\>^{k}dx.
\end{align*}
We may fix $\lambda$ such that 
$\lambda-\lambda_{0}(p)\geq 1$ and upon using Young's inequality, that is the fact that $ab \leq \epsilon a^p/p + b^{p'}/p'$ for $a,b \geq 0$, and choosing $a:=f_{0}$ and $b := \phi^{p-1}$, we conclude that we have
\begin{equation}
p(p-1)\intg |\gd \varphi|^{2}\varphi^{p-2}\, \zeta_{n}\,\<x\>^{k}dx + 
\intg \varphi^{p}\, \zeta_{n}\, \<x\>^{k}dx
\leq \intg |f_{0}|^p\, \zeta_{n}\,\<x\>^{k}dx ,
\end{equation}
It is clear now that letting $n$ tend to $\infty$, we deduce that $\varphi$, the 
solution of \eqref{eq:2.9}, belongs to $L^{p}_{k}$ and 
$$\|\phi\|_{L^p_{k}} \leq \|f_{0}\|_{L^p_{k}},$$
and that moreover
$$\intg |\gd \varphi|^{2}\varphi^{p-2}
\zeta_{n}\<x\>^{k}dx<\infty.$$ 

To finish the proof of the Lemma, when $f_{0} \geq 0$ belongs only to $L^p_{k}$ we 
consider a sequence $f_{0n}\in C^{\infty}_{c}$ such that $f_{0n} \geq 0$ and $f_{0n} \to f_{0}$ $L^{p}_{k}$ and we conclude by verifying easily that the corresponding solutions $\varphi_{n}$ converge to $\phi$ as $n\to\infty$. 

Indeed we have also $\|\phi\|_{L^p_{k}} \leq \|f_{0}\|_{L^p_{k}}$ and $L\phi \in L^p_{k}$, which means that the operator $L-\lambda I$ is $m$-dissipative on $L^p_{k}$. 
\end{proof}

Next we prove the following Nash type inequality which is going to be useful later.

\begin{Lem}\label{lem:Nash-k}
Let $f\in L^{1}_{k/2}(\R^{d})\cap 
H^{1}_{k}(\R^{d}),$ assume 
that $k>0$ when $d\geq2$ and $k\geq 2$ when 
$d=1.$ Then there exists a constant $C>0$ such 
that  the following inequality holds
\begin{equation}\label{eq:2.11}
\|f\|^{2+\frac{4}{d}}_{L^{2}_{k}}
\leq C \|f\|^{\frac{4}{d}}_{L^{1}
_{k/2}}\cdot \|\nabla f\|^{2}_{L^{2}_{k}}.
\end{equation}
\end{Lem}


\begin{proof} Let $f\in  L^{1}_{k/2}(\R^{d})
\cap H^{1}_{k}(\mathbb{R}^{d}),$
we  write 
$$\int_{\mathbb{R}^{d}}|f(x)|^{2} \langle x 
\rangle^{k}dx= \int_{\mathbb{R}^{d}}|f(x)\, \langle 
x \rangle^{\frac{k}{2}}|^{2} dx.$$
Therefore
$$\|f\|^{2+\frac{4}{d}}_{L^{2}
_{k}}= \|f \,\langle \cdot \rangle^{\frac{k}{2}}\|^{
2+\frac{4}{d}}_{L^{2}}.$$
Let us set 
$\varphi(x):=f(x)\langle x 
\rangle^{\frac{k}{2}},$
then by the Nash's classical inequality 
(J. Nash \cite{JFN})  
we have
$$ \|\varphi\|^{2+\frac{4}{d}}_{L^{2}}
\leq C
\|\varphi\|^{\frac{4}{d}}_{L^{1}
}\|\nabla \varphi\|^{2} _{L^{2}},$$
since 
$f\,\langle \cdot\rangle^{\frac{k}{2}} \in H^{1}.$
With simple calculations we can  see that 
$$\nabla\varphi =
\langle x \rangle^{\frac{k}{2}}\,\nabla f + 
\frac{k}{2}f(x)\,\langle x 
\rangle^{\frac{k}{2}-2}\,x.$$
Thus we have 
$$\|\nabla \varphi\|^{2}_{L^{2}}=
\int_{\mathbb{R}^{d}}|\nabla f|^{2} \langle x 
\rangle^{k}dx+ \frac{k^{2}}{4}
\int_{\mathbb{R}^{d}}|f|^{2}|x|^{2} \langle x 
\rangle^{k-4}dx.$$
$$+\frac{k}{2}
\int_{\mathbb{R}^{d}}2f(x)\nabla f(x) \cdot x 
\langle x \rangle^{k-2}dx$$
Integrating by parts the third
integral on the right hand side 
above, we obtain
$$\|\nabla \varphi\|^{2}_{L^{2}}=
\int_{\mathbb{R}^{d}}|\nabla f|^{2} \langle x 
\rangle^{k}dx- \frac{k}{2}
\int_{\mathbb{R}^{d}}f^{2}{\rm div}(\langle x 
\rangle^{k-2}\, x)dx + \frac{k^{2}}{4}
\int_{\mathbb{R}^{d}}|f|^{2}|x|^{2} \langle x 
\rangle^{k-4}dx.$$
Since we have 
$${\rm div}(\langle x \rangle^{k-2} x)=
(d+(d+k-2)|x|^2)\langle x \rangle^{k-4}.$$
We get
$$\|\nabla \varphi\|^{2}_{L^{2}}=
\|\nabla f\|^{2}_{L^{2}_{k}}+
\int_{\mathbb{R}^{d}}|f|^{2}\left[ 
\frac{k^{2}}{4} |x|^{2}- \frac{k}{2}(d+(d
+k-2))
|x|^{2}\right] \langle x \rangle^{k-4}dx.$$
Note that $$\left[ \frac{k^{2}}{4} |x|^{2}- 
\frac{k}{2}(d+(d+k-2))
|x|^{2}\right]=-\frac{k}{2}\left[ d+(d+
\frac{k}{2}-2)
|x|^{2}\right].$$
Therefore when 
$d\geq 2$ we have  $$d+\frac{k}{2}-2> 0$$ 
and 
$$\|\gd \varphi\|^{2}_{L^{2}}\leq 
\|\gd f\|^{2}_{L^{2}_{k}}$$
Otherwise if $d=1,$  we assume that $k\geq 2$
and thus, in this case also 
$$\|\gd \varphi\|^{2}_{L^{2}}\leq 
\|\gd f\|^{2}_{L^{2}_{k}}$$
Replacing 
$\varphi$ by $f\,\langle x 
\rangle^{\frac{k}{2}},$ we obtain the  lemma.
\end{proof}


\section{Existence of a stationary solution}

In this section we are interested in the 
existence and uniqueness of a stationary 
solution. To find this solution we want to 
use the Krein-Rutmann's theorem revisited by
J. Scher and S. Mischler \cite{JM}. For this 
we need some notions of Banach lattices, 
which we are going to recall. Let us consider 
the $L^{2}_{k}$ space equipped 
with its natural partial order $\geq.$ 
We set 
\begin{equation}
(L^{2}_{k})_{+}:=\big \{ f\;\in L^{2}_{k}\; ;\;
f \geq 0\big \}
\end{equation} 
\begin{Lem}[Kato's inequality]
\label{thm:Kat}
For all $f\in D(L)$ we have 
\begin{equation} 
\label{eq:Kato}
L|f|\geq {\rm sgn}(f)Lf,
\end{equation}
in the sense of distributions.
\end{Lem}
For more details on kato's inequality,  one can refer to W. Arendt \cite{WA},  T. Kato \cite{TK} R.Nagel and H. Ulig \cite{NU} and B. Simon\cite{BS}. For more convenience we give the following proof.
\medskip
 
\begin{proof}
Let $f\in D(L),$ by definition of $D(L)$ we 
have $f\in L^{2}_k$ and $Lf\in L^{2}_k.$
This implies in particular that $f\in H^1_{k}$, that is $\gd f\in L^{2}_{k}$
and we have also $\Delta f\in L^{2}_{{\rm loc}}$. Now
consider the function 
$$j_{\varepsilon}(s):=(\varepsilon^{2} + s^{2})^{1/2}-\varepsilon$$
 for $s\in \R.$
A simple calculation gives  
$$j'_{\varepsilon}(s)=\frac{s}
{\sqrt{\varepsilon^{2}+s^{2}}} \qquad 
\mbox{and}\qquad j''_{\varepsilon}(s)=
\frac{\varepsilon^{2}}{(\varepsilon^{2}+
s^{2})^{3/2}}.$$
Note that $j''_{\varepsilon}(s)\geq 0.$
For $f\in D(L)$  
we compute 
$\nabla j_{\varepsilon}(f)$ and 
$\Delta j_{\varepsilon}(f)$
that is
\begin{equation}
\nabla j_{\varepsilon}(f)= j'_{\varepsilon}(f)
\gd f
\end{equation}
and 
\begin{equation}
\Delta j_{\varepsilon}(f)=
j''_{\varepsilon}(f)|\gd f|^{2} + 
j'_{\varepsilon}(f)\Delta f.
\end{equation} 
Then for all $\varphi \in C^{\infty}_{c}$ such that $\varphi \geq 0$, since $j''(f)\geq 0$ we have 
\begin{equation}
\<\Delta j_{\varepsilon}(f), \varphi\>\geq
\<j'_{\varepsilon}(f)\Delta f,\varphi\>.
\end{equation}
Since $j_{\varepsilon}(f)\rightarrow |f|$ 
in $L^{2}_{k}$ and 
$\Delta \varphi \in L^{\infty}\cap L^{1},$ we 
have in particular
$$\Delta j_{\varepsilon}(f)\rightarrow
\Delta |f| \quad \mbox{in} \quad
\mathscr{D}'(\R^{d}).$$
Therefore $$\<\Delta j_{\varepsilon}(f),
\varphi\>\rightarrow \<\Delta |f|, \varphi\>.$$
On the other hand using the definition of $L$ we obtain
the following inequality
$$\<Lj_{\varepsilon}(f), \varphi\> \geq 
\<j'_{\varepsilon}(f)\Delta f + {\rm div}
(j_{\varepsilon}(f)\E), \varphi\>.$$ 
The right hand side term of the above 
inequality can be rewritten as follows:
$$\<j'_{\varepsilon}(f)\Delta f + {\rm div}
(j_{\varepsilon}(f\E)), \varphi\>=
\<j'_{\varepsilon}(f)\Delta f 
+j'_{\varepsilon}(f)\gd f\cdot \E + 
j_{\varepsilon}(f){\rm div}(\E), \varphi\>.$$
Due to the fact that $j'_{\varepsilon}(f) \rightarrow {\rm sgn}(f)$ a.e.\ on $\R^{d},$ one sees that
$$\<j'_{\varepsilon}(f)\Delta f,\varphi\>=
\intg j'_{\epsilon}(f)
\Delta f\varphi
dx\rightarrow \intg {\rm sgn}(f) \Delta f
\varphi dx.$$
On the other hand, since $\E$ and ${\rm div}(\E)$ belong to
$L^{\infty}_{\rm loc}(\R^{d}),$ and using the fact that $\gd f\cdot\E\in 
L^{2}_{{\rm loc}}$, as $\epsilon\to 0,$ we have 
$$\<j'_{\varepsilon}(f)\gd f\cdot \E, \varphi\>
=\intg j'_{\varepsilon}(f)\gd f\cdot\E dx
\rightarrow
\intg {\rm sgn}(f)\gd f\cdot \E\varphi dx.$$
So finally we infer that
$$\<j'_{\varepsilon}(f)\Delta f + j'_{\varepsilon}(f)\gd f\cdot \E, \varphi\>
\to \<{\rm sgn}(f)\Delta f 
+{\rm sgn}(f)\gd f\cdot \E , \varphi\>,$$
when $\varepsilon$ goes to $0$.
Analogously we have
$$\<j_{\varepsilon}(f){\rm div}(\E), 
\varphi\>=\intg j_{\varepsilon}(f) 
{\rm div}(\E)\varphi dx \rightarrow
\intg |f|{\rm div}(\E)\varphi dx,$$
when $\varepsilon$ goes to $0,$
since $|j_{\varepsilon}(f)|\leq |f|\in 
L^{2}_{{\rm loc}}.$
Now we remind that
$|s|={\rm sgn}(s)\cdot s.$ Therefore we have
$$\<|f|{\rm div}(\E) , \varphi\>=
\<{\rm sgn}(f)\;f{\rm div}(\E) , 
\varphi\>,$$  
and thus we obtain
$$\<L|f|,\varphi\>=\lim_{\varepsilon
\rightarrow 0}\<L j_{\varepsilon}(f),
\varphi\> \geq
\<{\rm sgn}(f)\Delta f+\gd |f|\cdot\E+|f|{\rm div}(\E), \varphi\>.$$
However one may check that
\begin{align*}
\gd |f|\cdot\E+|f|
{\rm div}(\E) & =
{\rm sgn}(f)\gd f\cdot \E 
+ {\rm sgn}(f)f {\rm div}(\E)  \\
&=  {\rm sgn}(f)\big[\gd f\cdot \E + f{\rm div}(\E)\big ] ,
\end{align*}
from which we conclude that
$${\rm sgn}(f)\Delta f+\gd |f|\cdot\E+|f| {\rm div}(\E) = {\rm sgn}(f)Lf,$$
and finally  that for all $\phi\in \mathscr{D}({\Bbb R}^d)$ such that $\phi \geq 0$ we have
$$\langle L|f|,\phi \rangle \geq \langle {\rm sgn}(f)Lf,\phi \rangle,$$
which is precisely the Kato's inequality \eqref{eq:Kato}. 
\end{proof}
The Kato's inequality will stay true if we 
replace $|f|$ by the positive part of $f$, that is $f^{+} := (|f|+f)/2.$ In this case we have
$$L f^{+}\geq {(1+{\rm sgn}(f)) \over 2}\,L f ={\mathbf{1}}_{\{f>0\}}\cdot Lf.$$

\begin{remark}
It is well known that if $L$ satisfies 
Kato's inequality then this is 
equivalent to say that  the semigroup 
which is generated by $L$ is a positivity preserving semigroup, 
in the sense that if $f_{0}\geq 0$ 
then $S_{L}(t)f_{0}\geq 0$. See for instance
B. Simon \cite{BS}, R. Nagel 
 and H. Uhlig  
 \cite{NU} [theorem 4.1 
 page 121])  
  or W. Arendt (\cite{WA}[ 
 theorem 1.6 page 159]). 

\end{remark}

\begin{remark}
 The Kato's inequality implies also the weak 
 maximum principle: in other words, if 
  $f\in D(L)$ and 
 $Lf\leq 0,$ then $f\geq 0.$ 
\end{remark}

\begin{Lem}[Strong maximum principle]
\label{thm:Str-MP}
Let $f\in D(L)\cap W^{2,\infty}(\R^{d})$, 
then the linear operator 
$L$ satisfies a strong maximum 
priciple. i.e
\begin{equation} 
\big (f\geq 0, \quad f\not\equiv 0\quad \mbox{and} \; \; Lf
\leq 0\big )\qquad \Rightarrow\qquad f>0.
\end{equation} 
\end{Lem}

\begin{proof}
Let $f\in D(L)$ be such that $f \not\equiv 0$ and $Lf \leq 0$ on ${\Bbb R}^d$. By the weak maximum principle, which is a consequence of Kato's inequality Lemma \ref{thm:Kat}, we know that $f \geq 0$, and actually for any $R > 0$ we have 
$$M(R) := \inf_{|x| \leq R} f(x) \geq 0.$$
If there exists $R_{0} > 0$ and $x_{0} \in B(0,R_{0})$ such that $f(x_{0}) = 0$, then $M(R_{0}) = 0$, and consequently for any $R > R_{0}$ we have also $M(R) = 0$. Therefore, according to Hopf maximum principle (see for instance Theorem 5, chapter 2, section 3 of the classical book of M.H.~Protter \& H.F.~Weinberger \cite{Protter-Weinberger}) we have $f \equiv 0$ in $\overline{B}(0,R)$, for all $R$, and this contradicts the fact that $f\not \equiv 0$ on ${\Bbb R}^d$.
\end{proof}

\begin{Lem}[Mass conservation for the semigroup]
\label{thm:M-Cons}
Let $(L, D(L))$ be defined by \eqref{eq:1.6}. Then we have the following identity.
\begin{equation}
L^{\ast}1=0, \qquad \mbox{in} 
\quad \mathscr{D}(\R^{d}),
\end{equation}
where $L^{\ast}$ is the formal adjoint:
$$L^{\ast}\varphi=\Delta\varphi-\E\cdot 
\gd\varphi, \quad \forall \; \varphi\; \in 
C^{\infty}_{c}(\R^{d}).$$
\end{Lem}
\qed

\begin{prop}
\label{thm:adj}
Let $(L,D(L))$ be defined by \eqref{eq:1.6} and. Assume that hypothesis (\textbf{H1})
holds. Then there exists $b\in \R$ and a 
function $\psi\in D(L^{\ast}),$  
$\psi>0,$ such that
\begin{equation}\label{eq:3.9}
L^{\ast}\psi\geq   b \psi,\quad \mbox{in}\quad
\mathscr{D}'(\R^{d}).
\end{equation} 
\end{prop}
\medskip 

\begin{proof}
Let $\alpha_{0} > 0$ and consider the function
$$\psi(x) := \langle x \rangle^{-\alpha_{0}} = (1 + |x|^2)^{-\alpha_{0}/2}.$$
One checks that if $\alpha_{0}$ is large enough, then $\psi \in D(L^*)$. 
Using the expressions \eqref{eq:nabla-x} and \eqref{eq:laplacian-x}, where $k$ is replaced with $-\alpha_{0}$, one checks that
$${L^*\psi \over \psi} = {\alpha_{0}(\alpha_{0} + 2 -d)|x|^{2} - \alpha_{0}d \over (1 + |x|^2)^2 } + {\alpha_{0}\, x\cdot {\bf E} \over 1 + |x|^2}.$$
Now using the fact that according to \eqref{eq:1.3} the function $x \mapsto x\cdot {\bf E}$ has a growth more than $\alpha|x|^{\gamma} + \beta$ with $\gamma\leq 2$ it is clear that
$${L^*\psi \over \psi} \geq  b := \inf_{x\in {\Bbb R}^d} 
\left[{\alpha_{0}(\alpha_{0} + 2 -d)|x|^{2} - \alpha_{0} d \over (1 + |x|^2)^2 } + {\alpha_{0} (\alpha |x|^{\gamma} + \beta) \over 1 + |x|^2 }\right] > -\infty. $$
Thus we have proved the Proposition.
\end{proof}

\medskip

 We will set 
 \begin{align*}
 \omega(L) &:=\inf\big \{ b\in\R \;,\;
 L- bI\quad \mbox{is $m$-dissipative}\big \},\\
 s(L) &:=\sup \big \{ Re (z),z\in 
 \Sigma(L)\big \}
 \end{align*}
 and $\Sigma(L)$ is the spectrum of $L.$
\medskip

Now we will recall a result of  
S. Mischler and J. Scher
\cite[Theorem 4.3, page 39]{JM} which reads :

\medskip

\begin{theo} 
\label{thm:SM-JC}
We consider an operator $L,$ wich is a generator of a semigroup $S_{L}(t)$ on a Banach lattice of functions $\mathbf{X}$ and we assume that:

\medskip

\noindent{\bf (1)} 
we have $L=A+B,$
where $B$ is a bounded linear operator and $A$ is 
such that there exists $\tau\in \R,$ 
such that $A-\tau I$
is $m$-dissipative.

\medskip
\noindent{\bf (2)} 
There exist $b\in \R$ and 
$\psi\in D(L^{\ast})\cap\mathbf{X}_{+}\backslash \{0\},$ such that we have $$L^{\ast}\psi\geq  b\psi.$$

\medskip
\noindent{\bf (3)} 
$S_{L}$ is a positivity preserving semigroup.

\medskip
\noindent{\bf (4)}
$L$ satisfies the strong maximum principle.\\
\medskip

Then we have $s(L)=\omega(L),$ and denoting this common value by
$\lambda,$ there exists $G\in D(L)$ such that 
$$G>0,\qquad LG=\lambda G\qquad\mbox{with}\quad\lambda:=s(L).$$

\end{theo}

\medskip
 The originality of this theorem is the fact that,
 it establishes a spectral theory result like Krein-Rutmann's in a non compact framework. It allows us to circumvent the lack of compactness 
of the linear operator. It keeps the philosophy of 
Krein-Rutman theorem while weakening itsassumptions.

\medskip

\begin{proof}(of Theorem \ref{thm:GS})\\
Let $(L,D(L))$ be defined by \eqref{eq:1.6}. From Theorem \ref{thm:Exist-SG} applied to the operator $L$ we 
have the existence of a real number $\lambda,$ such 
that $L-\lambda I$ be $m$-dissipative, and generates 
a semigroup on the space $\textbf{X}:=L^{2}_{k}.$ 
We are going to verify that the conditions of
Theorem \ref{thm:SM-JC} are satisfied.

\medskip
\noindent{(1)} We know that the operator $L$ can be split in the following way: for all 
$f\in D(L),$
$$Lf:= Bf + Af,$$ where $B$ is a bounded operator, and defined as follows there exists
$M > 0$ and $n \geq 1,$ such that for 
$\zeta_{n}$ gigen by \eqref{eq:Def-zeta-n}  
$$Bf := M\zeta_{n} f \quad \mbox{and}\quad 
f\in D(L).$$
And the linear operator $A$ is such that there
exists a real number  $\tau$, such that $A-\tau$ is $m$-dissipative and
$$Af := Lf - Bf,\quad \mbox{for all} \quad 
f\in D(L).$$

\medskip
The Proposition \ref{thm:adj}  applied to $L$ 
leads to the following:

\medskip
\noindent{\bf (2)}
There exists $b_{0}\in \R,$ such that for some
$b>b_{0},$ we have  a function $\psi>0,$
$\psi \in D(L^{\ast}),$ such that 
$$L^{\ast}\psi\geq b\psi.$$ 

\medskip

An application of Lemma \ref{thm:Kat} to the operator $L$ allows us
to deduce that $L$ satisfies Kato's inequality.
This means that:

\medskip

\noindent{\bf (3)} For all $f\in D(L)$ we have 
$L|f|\geq {\rm sgn}(f)Lf.$ Therefore the 
semigroup generated by 
$(L, D(L))$ is positivity preserving semigroup.

\medskip
Using Lemma \ref{thm:Str-MP}  one can assert that the linear
operator $L$ satisfies the strong  maximum principle. i.e

\medskip

\noindent{\bf (4)} If $(f\not \equiv 0,\; f\geq 0
\;\mbox{and}\; 
Lf\leq 0),$ then $f>0$ in $\R^{d}.$ 

\medskip
Then applying Theorem \ref{thm:SM-JC}
we conclude that , there exists 
$G>0,$ such that $$LG=0.$$  
This completes the proof of the Theorem \ref{thm:GS}.
\end{proof} 

\section{Exponential stability}
In this section we want to prove 
Theorem \ref{thm:Cv-GS}. And this proof will be 
based on the 
decomposition of the operator $L = A + B$, with a regular bounded operator $B$ and 
a linear operator $A$, such that  $A-\tau I$ 
is $m$-dissipative.
Before starting the proof of 
Theorem \ref{thm:Cv-GS},
we state the following results which will be 
useful for the sequel.
\begin{prop}
\label{thm:Contract}
Let $f_{0}\in D(L)$ be the initial data of
\eqref{eq:1.1}. Assume that hypotheses (\textbf{H0})---(\textbf{H3}) hold for $p=2.$
For $M > 0$ and $n \geq 1,$ let
$\zeta_{n}$ be gigen by \eqref{eq:Def-zeta-n}, set 
$$Bf := M\zeta_{n} f \quad \mbox{and}\quad\mbox{for all} \quad f\in D(L).$$ 
\begin{equation}
\label{eq:4.1}
Af := Lf - Bf,\quad \mbox{for all} \quad 
f\in D(L)
\end{equation}
Then we may fix $n\geq 1$ and $M > 0$ large enough so that there exists $\omega_0>0$ satisfying the following property: for all $f_{0} \in L^2_{k}$ we have
\begin{equation}
\label{eq:4.2}
\|S_{A}(t)f_{0}\|_{L^{2}_{k}}\leq 
e^{-\omega_0 t}\|f_{0}\|_{L^{2}_{k}}.
\end{equation}
\end{prop}
\medskip

\begin{proof}
Let $f_{0}\in D(L).$ We consider the following 
equation 
\begin{equation}
\label{eq:4.3}
 \Big \{ \begin{array}{l}
 \partial_{t} f - Af=0 
\\
f(0,x)=f_{0}(x)
 \end{array}  .
\end{equation}
There exists $\tau\in \R$ 
such that  $A-\tau I$  is 
$m$-dissipative, according to 
the definition of the 
linear operator $A$. Then there exists 
$f\in C^{1}
([0,T];L^{2}_{k})$ solution of 
\eqref{eq:4.3}. To simplify notations, 
let us set 
$f(t,x):=S_{A}(t)f_{0}(x).$ 
As the function $f\in 
C^{1}([0,T];L^{2}_{k})$ it makes 
sense to write 
$$\frac{d}{dt}\intg |f(t, x)|^{2}
\<x\>^{k}dx
=\intg (Af(t, x))f(t,x)
\<x\>^{k}dx$$
Using \eqref{eq:4.1}, we have
$$\intg (Af) f\<x\>^{k}dx=
\intg Lf f\<x\>^{k}dx-M\intg f^{2} \zeta_{n}
\<x\>^{k}dx.$$
Simple computations give us the following 
formula
\begin{align*}
\frac{d}{dt}\intg f^{2}\<x\>^{k}dx&=
-2\intg |\gd f|^{2}\<x\>^{k}dx +
\intg f^{2}\Delta\<x\>^{k}dx\\
&+
\intg f^{2}{\rm div}(\E)\<x\>^{k}dx
-\intg f^{2}\E\cdot \gd\<x\>^{k}
dx\\
&- M\intg f^{2}\zeta_{n}\<x\>^{k}dx.
\end{align*}
Using the fact that
$\intg |\gd f|^{2}\<x\>^{k}dx$ and
$M\intg f^{2}\zeta_{n}\<x\>^{k}dx$ are 
positive, one can deduce that
\begin{equation*}
\frac{d}{dt}\intg f^{2}\<x\>^{k}dx \leq 
\intg f^{2}\Delta\<x\>^{k}+ 
\intg f^{2}{\rm div}(\E)\<x\>^{k}dx - \intg f^{2}\E\cdot 
\gd\<x\>^{k}dx.
\end{equation*}
Using the expressions \eqref{eq:nabla-x} and \eqref{eq:laplacian-x}, we can write
\begin{align*}
\frac{d}{dt}\intg f^{2}\<x\>^{k}dx &\leq
\intg f(x)^{2}\frac{kd+k(k+d-2)|x|^{2}}{
(1+|x|^{2})^{2}}\<x\>^{k}dx\\
&+
\intg f^{2}{\rm div}(\E)\<x\>^{k}dx
-\intg f(x)^{2}
\frac{kx\cdot\E}{1+|x|^{2}}\<x\>^{k}dx.
\end{align*}
Then hypothesis (\textbf{H3}) allows us 
to write that 
$$\frac{d}{dt}\intg f^{2}\<x\>^{k}dx\leq
-2\omega^{\star}\intg f^{2}\<x\>^{k}dx.$$
Integrating in time between $0$ and $t,$ or
using Gronwall's lemma, we find that
$$\intg f^{2}\<x\>^{k}dx\leq \|f_{0}\|^{2}
_{L^{2}_{k}}e^{-2\omega^{\star}t}.$$ Thus
$$\|S_{A}(t)f_{0}\|^{2}_{L^{2}_{k}}\leq 
\|f_{0}\|_{L^{2}_{k}}e^{-2\omega^{\star}t}.$$
This completes the proof of the proposition.
\end{proof}

\medskip

For the reader's convenience we give the proof wich follows S. Mischler-J. Scher \cite{JM}.
\begin{proof}(of Theorem \ref{thm:Cv-GS}).
By iterating the formula \eqref{eq:1.2}, one has
\begin{equation}
\label{sdp}
S_{L}(t)=S_{A}(t)+ (S_{A}\ast BS_{A})(t) +
(S_{L}\ast BS_{A}\ast BS_{A})(t).
\end{equation}
Since $S_{L}(t)$ is a $\mathcal{C}_{0}$-semi group, 
there exists $\omega_{1}\in \R$ and $C_{0}\geq 1,$
such that $\|S_{L}(t)\|\leq C_{1}e^{\omega_{1} t}.$
Now we choose $a_{1}$ a real number, such that $a_{1} >\omega_{1}.$ 
Thus using the inverse Laplace transform formula we obtain the following representation: 
$$(S_{L}\ast BS_{A}\ast BS_{A})(t):=
\frac{1}{2i\pi}\int^{a_{1}+i\infty}_{
a_{1}-i\infty}e^{zt}R_{L}(z)(BR_{A}(z))^{2}dz.$$
Since $0$ is a simple eigenvalue of $L$ one can 
define the projection operator $\Pi$ on the space generated by $G,$ wich is the eigenfunction associated to $0.$ The projection operator $\Pi$ is defined as follows: for all $f\in D(L)$ 
$$\Pi f=M(f)G,\quad \mbox{where}\quad  M(f):=\intg f(x) dx.$$
Using Proposition \ref{thm:Contract} we can easily establish that 
\begin{equation}
\label{eq:sdp1}
\|(S_{A}\ast BS_{A})(t)\|\leq C^{2}\|B\|\;t\;e^{-\omega^{\star}t}.
\end{equation}
It has been shown in \cite[Theorem 2.1]{JM} that if 
there exist two linear operators $A$ and $B$ such that
$L=A+B,$ where $A$ and $B$ are given as  in Proposition \ref{thm:Contract} and such that \eqref{eq:sdp1}
holds, then a spectral gap exists, that is we 
can find $a^{\star}>0$ such that the spectrum 
$$\Sigma(L)\subset \big \{ z\;\in\; \mathbb{C}\;| \; Re (z)<-a^{\star}\;\big \}\cup\big \{ 0\big \}.$$
Now we choose $a$ such that $0< a < a^{\star}$ and we define 
$$J_{a}(t):=\frac{1}{2i\pi}\int^{-a+i\infty}_{-a-i\infty} e^{zt}R_{L}(z)(BR_{A}(z))^{2}dz.$$

Using \eqref{sdp} we have 
\begin{align*}
S_{L}(t)(I-\Pi)&=S_{A}(t)(I-\Pi)+ (S_{A}\ast BS_{A})(t)(I-\Pi)\\
&+\frac{1}{2i\pi}\int^{-a +i\infty}_{-a -i\infty}e^{zt}R_{L}(z)(BR_{A}(z))^{2}(I-\Pi)dz,
\end{align*}
Choose $\omega_{2}\in 
(a, a^{\star}),$ and consider  
$z=-a+is,$ for $s\in \R.$ Since 
$$(zI-A)R_{A}(z)=I,$$ then  
$$R_{A}=\frac{1}{z}(I-AR_{A}(z))=\frac{1}{-a+is}
(I-AR_{A}(-a+is)).$$ Thus
we have
\begin{align*}
\|J_{a}(t)\|&\leq
C\int^{+\infty}_{-\infty}|e^{zt}|\|R_{L}(-a+is)\| \|(BR_{A}(-a+is))\|^{2}ds \\
&\leq
C e^{-a\;t}
\int^{+\infty}_{-\infty}\|R_{L}(-a+is)\| \|BR_{A}(-a+is)\|^{2}ds\\
&\leq
C e^{-a\;t}
\int^{+\infty}_{-\infty}\frac{1}{(a^{2}+s^{2})}\|R_{L}(-a+is)\| \|B(I-AR_{A}(-a+is))\|^{2}ds.
\end{align*}
We know that since the operator $B(I-AR_{A}(-a+is))$ is bounded uniformly in $s$, then there 
exists a constant $C>0$ such that 
$$\forall s\in \R,\qquad  \|B(I-AR_{A}(-a+is))\|\leq C.$$ 
Since $L$ generates a $\mathcal{C}_{0}$-semigoup. One has 
$$\|R_{L}(-a+is)\|\leq \frac{C_{0}}{\omega_{2}-a}.$$  Consequently we obtain
$$\|J_{a}(t)\|\leq C e^{-a\;t}
\frac{C_{0}}{-a+\omega_{2}}
\int^{+\infty}_{-\infty}\frac{1}{(a^{2}+s^{2})}ds,$$ as $t\rightarrow +\infty,$
therefore  we have 
\begin{equation}
\label{eq:ja}
\|J_{a}(t)\|\leq\frac{C_{0}}{\omega_{2}-a}e^{-at}.
\end{equation}

Using the identity 
\eqref{sdp} we have 
\begin{align*}
S_{L}(t)(I-\Pi)f_{0} &= 
S_{A}(t)(I-\Pi)f_{0} + 
(S_{A}\ast BS_{A})(t)(I-\Pi)f_{0}\\
& +
(S_{L}\ast BS_{A}\ast BS_{A})(t)(I-\Pi)f_{0}.
\end{align*}
 Thus we have
 \begin{align*}
\| S_{L}(t)(I-\Pi)f_{0}\|_{L^{2}_{k}} &\leq 
\|S_{A}(t)(I-\Pi)f_{0}\|_{L^{2}_{k}}+
\|(S_{A}\ast BS_{A})(t)(I- \Pi)f_{0}\|_{L^{2}_{k}}\\
&+\|(S_{L}\ast BS_{A}\ast BS_{A})(t)(I-\Pi)f_{0}\|_{L^{2}_{k}}.
\end{align*}
Since we know that 
$$\|(S_{L}\ast BS_{A}\ast BS_{A})(t)(I-\Pi)f_{0}\|_{L^{2}_{k}}\leq \|J_{a}(t)\|\cdot 
\|(I-\Pi)f_{0}\|_{L^{2}_{k}}$$ and  that thanks to
\eqref{eq:ja} we have 
$$\|(S_{L}\ast BS_{A}\ast BS_{A})(t)(I-\Pi)f_{0}\|_{L^{2}_{k}}\leq \frac{C_{0}}{\omega_{2}-a}e^{-at}
\cdot 
\|(I-\Pi)f_{0}\|_{L^{2}_{k}}.$$
On the other hand we have prove that
$$\|(S_{A}\ast BS_{A})(t)(I-\Pi)f_{0}\|_{L^{2}_{k}}
\leq C\|B\|\;t\;e^{-\omega^{\star}t}
\|(I-\Pi)f_{0}\|_{L^{2}_{k}}.$$ 
We may use Proposition~\ref{thm:Contract} to obtain that 
$$\|S_{A}(t)(I-\Pi)f_{0}\|_{L^{2}_{k}}\leq
Ce^{-\omega^{\star}t}\|(I-\Pi)f_{0}\|_{L^{2}_{k}}.$$
Now we choose $0<\omega < \min(w^{\star},a)$ and then  
we have 
$$\|S_{L}(t)(I-\Pi)f_{0}\|_{L^{2}_{k}}\leq  
C(\omega,\;t)e^{-\omega\;t}\|(I-\Pi)f_{0}\|_{L^{2}_{k}}.$$
This completes the proof of Theorem \ref{thm:Cv-GS}.
\end{proof}


\end{document}